\documentclass[a4paper]{amsart}
\usepackage{amssymb, amsfonts, amsmath, amsthm, graphicx, bm, tikz, enumerate}

\theoremstyle{plain} \newtheorem{theorem}{Theorem}[section]

\newtheorem{lemma}[theorem]{Lemma}

\newtheorem{corollary}[theorem]{Corollary}
\theoremstyle{remark} \newtheorem{remark}[theorem]{Remark}
\newtheorem{eg}[theorem]{Example}
\theoremstyle{definition} \newtheorem{definition}[theorem]{Definition}

\newcommand{\J}{\mathcal{J}_\beta} 
\newcommand{\Tb}{T_\beta}

\newcommand{\gb}{\gamma(\beta)}
\newcommand{\Xb}{\mathfrak{X}_\beta} 
\newcommand{\infXgb}{\inf X_{\gamma(\beta)}}
\newcommand{\bgb}{\bm{\gamma}(\beta)} 
\newcommand{\1}{1_{\beta}} 
\newcommand{\rr}{\mathbf{r}} 

\newcommand{\Q}{\mathbb{Q}}
\newcommand{\N}{\mathbb{N}}
\newcommand{\R}{\mathbb{R}}

\title{The $\beta$-transformation with a hole}
\date{\today}
\address{School of Mathematics\\
The University of Manchester\\
Oxford Road, Manchester, M13 9PL, UK}
\author{Lyndsey Clark}
\email{Lyndsey.Clark@manchester.ac.uk}
\begin{document}
 
\begin{abstract}
 This paper extends those of Glendinning and Sidorov \cite{glendinningsidorov15} and of Hare and Sidorov \cite{haresidorov14} from the case of the doubling map to the more general $\beta$-transformation. Let $\beta \in (1,2)$ and consider the $\beta$-transformation $\Tb(x)=\beta x \pmod 1$. Let $\J(a,b) := \{ x \in (0,1) : \Tb^n(x) \notin (a,b) \text{ for all } n \geq 0 \}$. An integer $n$ is bad for $(a,b)$ if every periodic point of period $n$ for $\Tb$ intersects $(a,b)$. Denote the set of all bad $n$ for $(a,b)$ by $B_\beta(a,b)$. In this paper we completely describe the following sets:
\begin{align*}
 D_0(\beta) &= \{ (a,b) \in [0,1)^2 : \J(a,b) \neq \emptyset \}, \\
D_1(\beta) &= \{ (a,b) \in [0,1)^2 : \J(a,b) \text{ is uncountable} \}, \\
D_2(\beta) &= \{ (a,b) \in [0,1)^2 : B_\beta(a,b) \text{ is finite} \}.
\end{align*}
\end{abstract}

\maketitle

\section{Introduction}
Consider $\beta \in (1,2)$ and let $\Tb: [0,1) \to [0,1)$ denote the $\beta$-transformation, that is $\Tb(x) = \beta x \pmod 1$. Our main object of study is the \emph{avoidance set} of a hole:
\begin{equation*}
 \J(a,b) = \{ x \in (0,1) : \Tb^n(x) \notin (a,b) \text{ for all } n \geq 0 \},
\end{equation*}
where $0 < a < b < 1$. This is the set of points whose orbits are disjoint from the ``hole'' $(a,b)$. The map $\Tb$ restricted to $\J(a,b)$ is what is referred to as an \emph{open map}, or a \emph{map with a hole}. Intuitively, if $(a,b)$ is small then $\J(a,b)$ should be large, and vice versa. This paper aims to generalise results about this set from the case of the doubling map to the more general $\beta$-transformation for $\beta \in (1,2)$.

An integer $n$ is bad for $(a,b)$ if every periodic point of period $n$ for $\Tb$ intersects $(a,b)$. Denote the set of all bad $n$ for $(a,b)$ by $B_\beta(a,b)$. Then define as follows:
\begin{align*}
 D_0(\beta) &= \{ (a,b) \in [0,1)^2 : \J(a,b) \neq \emptyset \}, \\
D_1(\beta) &= \{ (a,b) \in [0,1)^2 : \J(a,b) \text{ is uncountable} \}, \\
D_2(\beta) &= \{ (a,b) : B_\beta(a,b) \text{ is finite}\}.
\end{align*}

These sets were fully described for the doubling map by Glendinning and Sidorov \cite{glendinningsidorov15} ($D_0(2)$ and $D_1(2)$) and by Hare and Sidorov \cite{haresidorov14} ($D_2(2)$). This work showed that the structure of these sets is such that the boundaries are devil's staircases. In particular, specific 0-1 words known as balanced words are very important in these descriptions. We show that for $\beta \in (1,2)$ there is a similar devil's staircase structure, but that balanced words alone are insufficient to obtain the full description. Instead we need the more general \emph{maximal extremal pairs}, defined in Section~\ref{extrdefnsect}. In Section~\ref{transfersection} we transfer what results we can from the doubling map and extend them to the general case. Section~\ref{rbeta} describes the unbalanced maximal extremal pairs that are not relevant to the doubling map case and is completely new.

\section{$\beta$-expansions and combinatorics on words} \label{extrdefnsect}
Much of the study of avoidance sets involves combinatorics on words. We therefore include the basic definitions from combinatorics on words here --- see \cite[Chapter 2]{lothaire02} for a more thorough discussion. We will be considering words on the alphabet $\{0,1\}$. Given two finite words $u=u_1\dots u_n$ and $v=v_1\dots v_m$ we denote by $uv$ their concatenation $u_1 \dots u_n v_1 \dots v_m$. In particular $u^k = u \dots u$ ($k$ times) and $u^\infty = \lim_{k \to \infty} u^k$. We denote the length of $u$ by $|u|$ and the number of 1s in $u$ by $|u|_1$. For a finite or infinite word $w$, a finite word $u=u_1 \dots u_m$ is said to be a \emph{factor} of $w$ if there exists an $i$ such that $u_1 \dots u_m = w_{i+1} \dots w_{i+m}$.
To compare words we use the \emph{lexicographic order}: a finite or infinite word $u$ is lexicographically smaller than a word $v$ (that is, $u \prec v$) if either $u_1<v_1$ or there exists $k > 1$ with $u_i = v_i$ for $1 \leq i<k$ and $u_k < v_k$.

A finite or infinite word $w$ is said to be \emph{balanced} if for any two factors $u$ and $v$ of $w$ of equal length, we have that $||u|_1 - |v|_1| \leq 1$. A finite word $w$ is called \emph{cyclically balanced} if $w^2$ is balanced. Infinite aperiodic balanced words are commonly called \emph{Sturmian words}.

We also introduce the following notation. Given a finite word $w$ and a factor $u$ of $w$, we denote by $u$-$\max(w)$ the lexicographically maximal cyclic permutation of $w$ that begins with the word $u$. Similarly we denote by $u$-$\min(w)$ the lexicographically minimal cyclic permutation of $w$ that begins with the word $u$. For example, given $w=10100$, we have 0-$\max(w) = 01010$ and 1-$\min(w) = 10010$.

In order to use combinatorics on words in the context of the $\beta$-transformation, we recall that $\Tb$ is conjugate to the shift map on a subset of $\Sigma = \{0,1\}^\N$. This arises by writing a number $x$ as
\begin{equation*}
 x = \sum_{i \geq 1} x_i \beta^{-i},
\end{equation*}
with $x_i \in \{0, 1\}$, as first studied in \cite{renyi57}. In particular, we consider the \emph{greedy} $\beta$-expansion of $x$, namely the expansion with $x_i = \lfloor \beta T_\beta^{i-1} x \rfloor$. Informally the greedy expansion is given by taking a $1$ whenever possible. We denote the set of possible (``admissible'') greedy sequences $(x_i)_{i=1}^\infty$ by $\Xb$.

Consider the expansion of $1$ given by $\tilde{d}_i = \lfloor \beta T_\beta^{i-1}(1) \rfloor$. If this sequence is infinite (i.e. does not end in $0^\infty$) then set $d_i = \tilde{d}_i$. If $\tilde{d}_i$ is finite then let $k = \max \{j : \tilde{d}_j \neq 0 \}$ and set $d_1 d_2 \dots = (\tilde{d}_1 \dots \tilde{d}_{k-1} 0 )^\infty$. This is the periodic \emph{quasi-greedy expansion} of $1$. Then as shown by Parry in \cite{parry60}, we have
\begin{equation*}
 \Xb = \{ (x_i)_{i= 1}^\infty : x_j x_{j+1} x_{j+2} \dots \preceq d_1 d_2 d_3 \dots \text{ for all } j \in \N \}.
\end{equation*}

We will denote the quasi-greedy expansion of $1$ for a given $\beta$ by $\1$.

\begin{eg}
 Consider $\beta$ given by the golden ratio $\phi=(1+\sqrt{5})/2$. Then the greedy expansion of $1$ is $110^\infty$ and the quasi-greedy expansion of $1$ is $\1=(10)^\infty$. Therefore for this value of $\beta$, the set $\Xb$ is the set of all 0-1 sequences that do not contain $11$ as a factor.
\end{eg}
As $\beta$ increases, $\1$ increases lexicographically. For example, large $\beta$ close to $2$ will have $\1 = 1^n 0 \dots$ for some large $n$. Small $\beta$ close to $1$ will have $\1 = 10^n 1 \dots$ for some large $n$.

Throughout this paper we will refer to a point $x \in (0,1)$ and its expansion $(x_i)_{i=1}^\infty \in \Xb$ interchangeably. The only possible ambiguity here is where a point has a finite expansion; that is to say $(x_i)=u10^\infty$ for some finite admissible word $u$. Here we naturally have $u10^\infty = u0\1$. Generally in such cases we will use the finite expansion by default and will specify if this is not the case.

We also make use of the idea of extremal pairs, linked to the study of Lorenz maps through kneading invariants -- see \cite{hubbardsparrow90} and \cite{glendinningsparrow93}. The essential idea is that given some orbit, we take two neighbouring points of that orbit, meaning that the rest of the orbit does not fall between these two points.

\begin{definition}[Extremal pairs\footnote{There are more general versions of this concept, as seen in \cite{hubbardsparrow90} and \cite{glendinningsparrow93}. These cover cases when $t$ is not a cyclic permutation of $s$ and cases involving infinite aperiodic orbits. However, these situations are not needed for our context and so we omit a more general definition.}]
 Let $(s,t)$ be a pair of finite $\{0,1\}$ words with $t$ a cyclic permutation of $s$ and $s^\infty \prec t^\infty$. Then $(s,t)$ is said to be an \emph{extremal pair} if for every $k \in \N$, either $\sigma^k s^\infty \preceq s^\infty$ or $\sigma^k s^\infty \succeq t^\infty$.
\end{definition}
Notice that we do not require that $s_1=0$ and $t_1=1$ as in \cite{glendinningsidorov15}. However as an immediate consequence of the definition we have that for every extremal pair $(s,t)$ there exist words $w$ and $u$ such that $u0$ and $u1$ are factors of $w$ and $s=u0$-$\max(w)$ and $t= u1$-$\min(w)$.

\begin{eg}
 Consider the periodic point $(1100)^\infty$. Extremal pairs arising from this orbit are $(0011,0110)$, $(0110,1001)$, and $(1001,1100)$. However the pair $(s,t)=(0110,1100)$ is \emph{not} extremal because $s^\infty \prec (1001)^\infty \prec t^\infty$.
\end{eg}

Given an extremal pair $(S,T)$, we have $S, T \in \{0,1\}^n$ and so denote more fully the pair as $(S(0,1), T(0,1))$. Then one can take another extremal pair $(s,t)$ and use this pair as an alphabet to gain the pair $(S(s,t), T(s,t))$. These words then belong to $\{s,t\}^n$. We call such a pair $(S(s,t), T(s,t))$ a \emph{descendant of $(s,t)$}. It is shown in \cite[Proposition 2.1]{glendinningsidorov15} that all such descendants are themselves extremal pairs.

\begin{definition}[Maximal extremal pairs]
 An extremal pair $(s,t)$ is said to be \emph{maximal} if firstly there does not exist any point $x$ such that the orbit of $x$ is contained in one of either $[0, s^\infty)$ or $(t^\infty, 1)$, and secondly there does not exist a distinct extremal pair $(\tilde{s}, \tilde{t})$ such that $(s^\infty, t^\infty) \subset (\tilde{s}^\infty, \tilde{t}^\infty)$.
\end{definition}

\begin{eg}
 For the doubling map, the extremal pair $(0110,1001)$ is \emph{not} a maximal extremal pair because $(01,10)$ is an extremal pair with
\[ (01)^\infty \prec (0110)^\infty \prec (1001)^\infty \prec (10)^\infty. \]
It is shown in \cite{glendinningsidorov15} that $(01,10)$ is an example of a maximal extremal pair for the doubling map.
\end{eg}

We make one further definition pertaining to avoidance sets:
\begin{definition}
 We say an avoidance set $\J(a,b)$ is \emph{essentially equal to} an invariant set $X$, written $\J(a,b) \doteq X$, if for every $y \in \J(a,b)$, there exists $x \in X$ and $w \in \{0,1\}^n$ with $y = wx$.
\end{definition}
The idea of this definition is that we would like to describe an avoidance set up to preimages (expressed by the word $w$), but we cannot possibly mean \emph{every} preimage, because some preimages will themselves fall into $(a,b)$. Therefore this definition implicitly comes with substantial restrictions on what words $w$ will be possible, because we cannot have $y=wx \in (a,b)$ or indeed $\sigma^ny \in (a,b)$ for any $n$.

\begin{lemma} \label{maxisnice}
An extremal pair $(s,t)$ is maximal if and only if
\[\J(s^\infty, t^\infty) \doteq \{ \sigma^n s^\infty  : n \in \N \}.\]
\end{lemma}
\begin{proof}
Suppose $(s,t)$ is maximal. Then by definition the only periodic orbit avoiding $(s^\infty, t^\infty)$ is $s^\infty$ itself. Suppose there exists some point $x$ that avoids the hole. Approximate $x$ by periodic points. All of these periodic points must fall into $(s^\infty, t^\infty)$. Therefore by taking limits we see that the only way the orbit of $x$ can avoid the hole is by falling onto the boundary. Therefore $x$ must be of the form $ws^\infty$ for some finite word $w$, where $\sigma^n w s^\infty \notin (s^\infty, t^\infty)$ for every $n$. This is precisely what it means to have $\J(s^\infty, t^\infty) \doteq \{ \sigma^n s^\infty \}$.
\end{proof}

\subsection{Admissible balanced sequences for the $\beta$-transformation}
For the case $\beta=2$ the maximal extremal pairs are formed from balanced words, as shown in \cite{glendinningsidorov15}. Hence, the first issue is to establish which balanced sequences are admissible for which $\beta$. A detailed exposition on balanced words may be found in the book by Lothaire \cite[\S 2]{lothaire02} and the survey paper by Vuillon \cite{vuillon03}. There are many ways of defining both finite cyclically balanced words and Sturmian sequences. We will do so using the Farey tree\footnote{Note the commonly defined infinite version of this tree is known as the Stern-Brocot tree.} as the tree structure allows for easier proofs later.

\begin{definition}[Farey tree]
We construct the Farey tree inductively. Take $0$ and $1$ as initial words, with associated fractions $0/1$ and $1/1$ respectively. Two fractions $a/b$ and $c/d$ with $a/b < c/d$ are said to be neighbours if $bc-ad=1$, and two words $w_{a/b}$ and $w_{c/d}$ are said to be neighbours if their associated fractions are neighbours.

Given two neighbouring words $w_{\gamma_1}$ and $w_{\gamma_2}$ such that $w_{\gamma_1}^\infty < w_{\gamma_2}^\infty$, combine the associated fractions $\gamma_1 = a_1/b_1$ and $\gamma_2 = a_2/b_2$ to make $\gamma_1 \oplus \gamma_2 = (a_1+a_2)/(b_1+b_2)$. Then we define $w_{\gamma_1 \oplus \gamma_2} = w_{\gamma_2}w_{\gamma_1}$.
\end{definition}
The beginning of the trees for both fractions and words are depicted in Figure~\ref{fareytreepic}.

Given $\gamma = \gamma_1 \oplus \gamma_2$ with $\gamma_1 < \gamma_2$ we will refer to $\gamma_1$ and $\gamma_2$ as the \emph{left} and \emph{right Farey parents} respectively, and $\gamma$ will be referred to as the \emph{child} of $\gamma_1$ and $\gamma_2$. We will also use these terms for the associated words where appropriate.

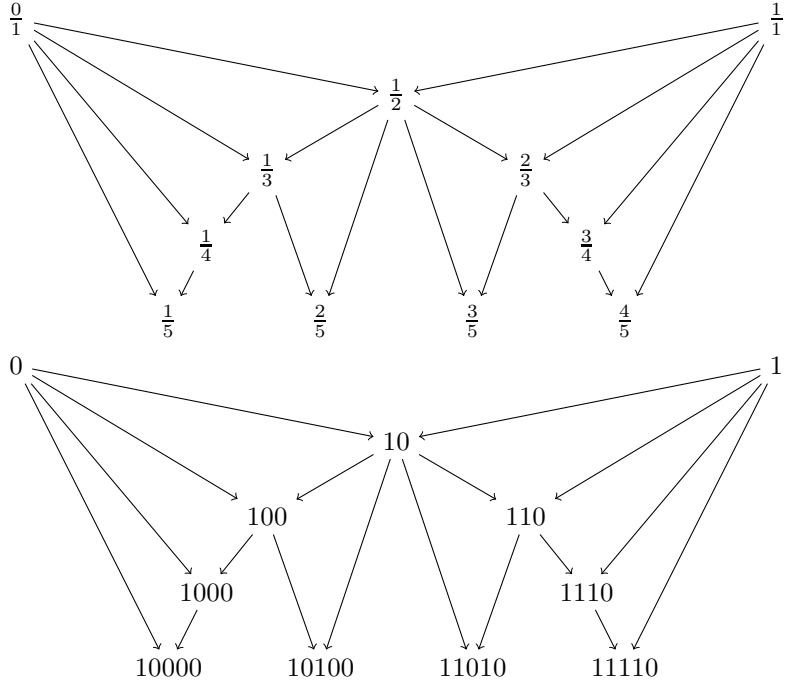
\begin{figure}
\centering
\begin{tikzpicture}[->]
\node (0) at (0,5) {$\frac{0}{1}$};
\node (1) at (10,5) {$\frac{1}{1}$};

\node (12) at (5,4) {$\frac{1}{2}$};

\node (13) at (3.3,3) {$\frac{1}{3}$};
\node (23) at (6.7,3) {$\frac{2}{3}$};

\node (14) at (2.5,2) {$\frac{1}{4}$};
\node (34) at (7.5,2) {$\frac{3}{4}$};

\node (15) at (2,1) {$\frac{1}{5}$};
\node (25) at (4,1) {$\frac{2}{5}$};
\node (35) at (6,1) {$\frac{3}{5}$};
\node (45) at (8,1) {$\frac{4}{5}$};

\path (0) edge (12);
\path (0) edge (13);
\path (0) edge (14);
\path (0) edge (15);

\path(1) edge (12);
\path(1) edge (23);
\path(1) edge (34);
\path(1) edge (45);

\path (12) edge (13);
\path (12) edge (23);

\path (13) edge (14);
\path (23) edge (34);

\path (12) edge (25);
\path (12) edge (35);

\path (13) edge (25);
\path (23) edge (35);

\path (14) edge (15);
\path (34) edge (45);

\end{tikzpicture}

\begin{tikzpicture}[->]
\node (0) at (0,5) {$0$};
\node (1) at (10,5) {$1$};

\node (12) at (5,4) {$10$};

\node (13) at (3.3,3) {$100$};
\node (23) at (6.7,3) {$110$};

\node (14) at (2.5,2) {$1000$};
\node (34) at (7.5,2) {$1110$};

\node (15) at (2,1) {$10000$};
\node (25) at (4,1) {$10100$};
\node (35) at (6,1) {$11010$};
\node (45) at (8,1) {$11110$};

\path (0) edge (12);
\path (0) edge (13);
\path (0) edge (14);
\path (0) edge (15);

\path(1) edge (12);
\path(1) edge (23);
\path(1) edge (34);
\path(1) edge (45);

\path (12) edge (13);
\path (12) edge (23);

\path (13) edge (14);
\path (23) edge (34);

\path (12) edge (25);
\path (12) edge (35);

\path (13) edge (25);
\path (23) edge (35);

\path (14) edge (15);
\path (34) edge (45);

\end{tikzpicture}
\caption{The Farey tree and corresponding balanced words}
\label{fareytreepic}
\end{figure}

For rational $\gamma=p/q$, it is well known (see Lothaire \cite[\S 2]{lothaire02}) that the word $w_\gamma$ is the unique (up to cyclic permutation) cyclically balanced word of length $q$ containing $p$ 1s. Define $X_{p/q}$ to be the finite set
\begin{equation*}
 X_{p/q} = \{ \sigma^n (w_\gamma^\infty) : n \in \N \}.
\end{equation*}
These sets $X_{\gamma}$ has been well studied in \cite{bullettsentenac94}. For rational $\gamma$ the word $w_\gamma$ given by the Farey tree turns out to be the maximal cyclic shift; that is to say $w_\gamma^\infty$ is the maximal element of $X_\gamma$.

The limit points of the tree correspond to irrational $\gamma$, and give rise to the infinite aperiodic balanced words known as Sturmian sequences. In this case the set $X_\gamma$ is a Cantor set with Hausdorff dimension $0$ and $w_\gamma$ is the supremum of $X_\gamma$ (\cite{bullettsentenac94}).

Each rational $\gamma$ gives rise to two distinct infinite balanced words. Given $w_\gamma = w_{\gamma_2}w_{\gamma_1}$, these words are $w_\gamma^\infty$ and $w_{\gamma_2} w_\gamma^\infty$.

Therefore we may define a function as follows:

\begin{definition}[$\gb$]
 Define $ \gamma: (1,2) \to (0,1)$ to be the number associated to the maximal admissible balanced word for a given $\beta \in (1,2)$.
\end{definition}
Naturally we have that if $\gamma_1 < \gamma_2$ then $w_{\gamma_1} \prec w_{\gamma_2}$. Therefore $\gb$ is non-decreasing, with the effect that for a given $\beta$, $w_\gamma^\infty$ is admissible if and only if $\gamma \leq \gb$.

We can describe this function using the following lemma:
\begin{lemma}[Admissible balanced words] \label{gbdescription} We have $\gamma(\beta) = \gamma \in \Q$ for every $\beta$ such that $\1 \in [w_\gamma^\infty, w_{\gamma_2} w_\gamma^\infty]$, where $\gamma_2$ is the right Farey parent of $\gamma$.
\end{lemma}
\begin{proof}
Firstly notice that as $w_\gamma^\infty$ is the maximal element of $X_\gamma$, it is clear that $w_\gamma^\infty$ is admissible if and only if $\1 \succeq w_\gamma^\infty$, and so $\gb \geq \gamma$ if and only if $\1 \succeq w_\gamma^\infty$.

To see the result for the right endpoint, consider the sequence of rationals $r_n = \gamma_2 \oplus \gamma \oplus \dots \oplus \gamma$ ($n$ times). Then $r_n \searrow \gamma$ and $w_{r_n}^\infty = (w_{\gamma_2} w_\gamma^n)^\infty \searrow w_{\gamma_2} w_\gamma^\infty$. Therefore, if $\1 \succ w_{\gamma_2} w_\gamma^\infty$ then there exists some $n$ such that $w_{r_n}^\infty$ is admissible, meaning that $\gamma(\beta) \geq r_n > \gamma$.
\end{proof}
By a similar argument it is easy to see that $\gamma(\beta)$ takes an irrational value $\gamma$ precisely when $\1 = w_\gamma$.

\begin{eg} For example, $\gb = 1/2$ for every $\beta$ such that $\1 \in [(10)^\infty, 1(10)^\infty]$. This corresponds to $\beta \in [\varphi, 1.8019\dots]$ where $\varphi$ denotes the golden ratio. This is the largest plateau visible in Figure~\ref{staircase}.
\end{eg}

\begin{figure}
\centering
 \includegraphics[width=0.8\textwidth]{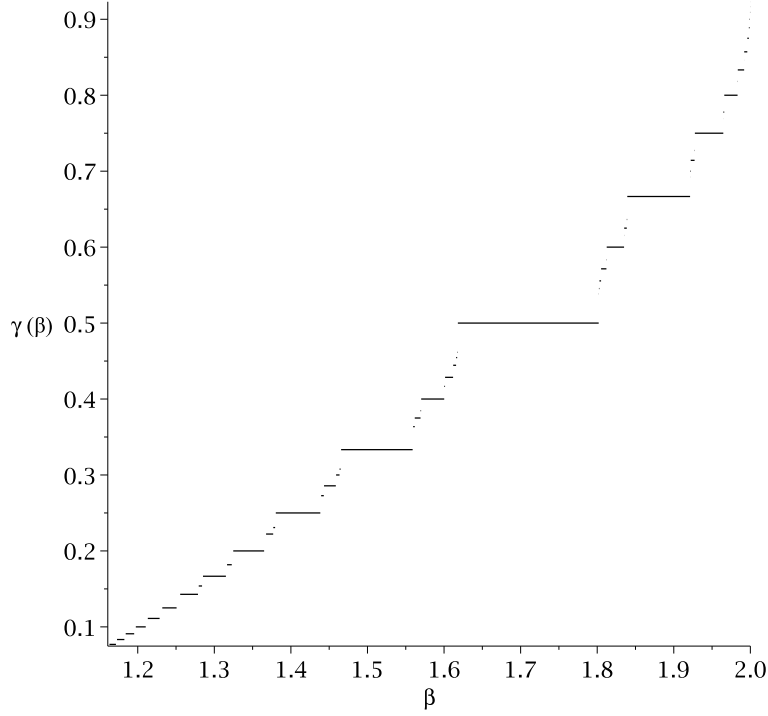}
\caption{$\gamma(\beta)$, a devil's staircase.}
\label{staircase}
\end{figure}

One consequence of Lemma~\ref{gbdescription} is that the intervals $[w_\gamma^\infty, w_{\gamma_2} w_\gamma^\infty]$ are disjoint for distinct $\gamma \in \Q$. As can be seen in Figure \ref{staircase}, $\gb$ is a devil's staircase: it is continuous, non-decreasing, and has zero derivative almost everywhere. The same function arises when considering digit frequencies for $\beta$-expansions, as described by Boyland et al.\ in \cite{boylandetal13}.

We can associate an extremal pair with each rational $\gamma$. Given the word $w_\gamma$, take $s_\gamma=0$-$\max (w_\gamma)$ and $t_\gamma=1$-$\min (w_\gamma)$. For example, $s_{2/5}=01010$ and $t_{2/5} = 10010$.

These pairs may themselves be constructed using a tree structure: given two neighbouring pairs $(s_{\gamma_1}, t_{\gamma_1})$ and $(s_{\gamma_2}, t_{\gamma_2})$ with associated rationals $0<\gamma_1 < \gamma_2<1$, their child is $(s_{\gamma_2} s_{\gamma_1}, t_{\gamma_1} t_{\gamma_2})$.\footnote{Note the reason for excluding $\gamma = 0$ or $1$ is that here $s$ and $t$ are not well defined.}

\begin{remark} Balanced pairs satisfy the following:
\begin{align*}
 s_{\gamma_1} t_{\gamma_1}^\infty &< (s_{\gamma_2} s_{\gamma_1})^\infty, \\
s_{\gamma_2}s_{\gamma_1}(t_{\gamma_1}t_{\gamma_2})^\infty &< s_{\gamma_2}^\infty.
\end{align*}
This means that the intervals $[s_\gamma^\infty, s_\gamma t_\gamma^\infty]$ do not overlap.
\end{remark}
This is shown for balanced words in \cite[Lemma 3.2]{haresidorov14}, and may also be seen as a corollary of Lemma~\ref{gbdescription}. We remark that the result is actually a property of the tree construction method, not specifically of the words themselves: to obtain the same result for a different tree, we simply take the tree for balanced words and map $0$ and $1$ to the left and right roots of our alternative tree. Then provided the left root is less than the right root, the result will hold.

\subsection{Descendants of balanced words}
For describing $D_0(\beta)$, the normal balanced words are sufficient. However, to describe $D_1(\beta)$ one must also consider their descendants. These are defined in \cite{glendinningsidorov15} and for completeness' sake we repeat the discussion here. These words are not themselves balanced but are derived from balanced words.

Consider $p/q \in (0,1)$. Define a function $\rho_{p/q}: \{0,1\} \to \{0,1\}^q$ by
\begin{align*}
 \rho_{p/q}(0) &= s_{p/q}, \\
 \rho_{p/q}(1) &= t_{p/q},
\end{align*}
where $s=0$-$\max (w_{p/q})$ and $t= 1$-$\min (w_{p/q})$ are the balanced extremal pair associated to $p/q$ as defined in the previous section. We extend the definition of $\rho$ to finite words by concatenation: given some word $u=u_1 \dots u_n$ define $\rho_{p/q}(u_1 \dots u_n) = \rho_{p/q}(u_1) \dots \rho_{p/q} (u_n)$.

Then for $\rr \in (\Q \cap (0,1))^n$ we may define as follows:
\begin{align*}
s_{\rr} &= \rho_{r_1} \rho_{r_2} \dots \rho_{r_n} (0), \\
 t_{\rr} &= \rho_{r_1} \rho_{r_2} \dots \rho_{r_n} (1).
\end{align*}

\begin{eg} We have $s_{(2/5,1/2)} = \rho_{2/5} \rho_{1/2} (0) = \rho_{2/5}(01) = 0101010010$ and $t_{(2/5,1/2)} = \rho_{2/5}(10) = 1001001010$.
\end{eg}

By taking limits this definition may be extended to $\rr \in (\Q \cap (0,1))^\N$ and to $\rr \in (\Q^{n-1} \times \R) \cap (0,1)^n$.

\begin{remark} \label{fareydescs}
In \cite[Lemma 2.10]{glendinningsidorov15} it is shown that:
\begin{enumerate}
 \item Given $\rr = (r_1, \dots, r_n)$ and $\rr' = (r_1', \dots, r_n')$ with $\rr \neq \rr'$, we have that $[s_{\rr}^\infty, s_{\rr}t_{\rr}^\infty] \cap [s_{\rr'}^\infty, s_{\rr'}t_{\rr'}^\infty] = \emptyset$;
 \item Given $\rr = (r_1, \dots, r_n)$ and $\tilde{\rr} = (r_1, \dots, r_{n-1})$, we have that
\begin{equation*}
 [s_{\rr}^\infty, s_{\rr}t_{\rr}^\infty] \subset [s_{\tilde{\rr}}t_{\tilde{\rr}}s_{\tilde{\rr}}^\infty, s_{\tilde{\rr}}t_{\tilde{\rr}}^\infty].
\end{equation*}
\end{enumerate}
\end{remark}

This enables us to make the following definition.
\begin{definition}[$\bgb$]
 We define $\bgb$ to be the (finite or infinite) sequence of numbers corresponding to the maximal admissible balanced descendant. We write this as a vector:
\begin{equation*}
 \bgb \in \left( \bigcup_{n=1}^\infty (\Q \cap (0,1))^n \right) \cup \left( \bigcup_{n=1}^\infty (\Q^{n-1} \times \R) \cap (0,1)^n \right).
\end{equation*}
\end{definition}

Remark~\ref{fareydescs} enables us to note that the function $\gb$ defined in the previous section corresponds to $(\bgb)_1$. Essentially on each plateau of $\gb$, we define a new devil's staircase giving $(\bgb)_2$. Each plateau of this will then give rise to a further devil's staircase for $(\bgb)_3$, and so the process continues. Throughout this text we shall refer to $(\bgb)_1=\gb$: we will need the vector when discussing $D_1(\beta)$, but only the scalar is needed for $D_0(\beta)$ and $D_2(\beta)$.

\begin{eg}
 Consider the case $\1 = (10010000)^\infty$, $\beta \approx 1.427$. Here $(1000)^\infty$ is admissible but $100(1000)^\infty$ is inadmissible, therefore $\gb=(\bgb)_1 = 1/4$. Then consider $\rr=(1/4,1/2)$. This gives $(s_{\rr}, t_{\rr}) = (01001000,10000100)$. Therefore $s_{\rr}^\infty$ is admissible and indeed $\sigma s_{\rr}^\infty = \1$, meaning that $s_{\rr}^\infty = 1/\beta$. It follows that $s_{\rr} t_{\rr}^\infty$ is inadmissible. Therefore the maximal admissible descendant of balanced words is given by $\bgb = (1/4, 1/2)$.
\end{eg}

We can use the above definitions in the more general setting of maximal extremal pairs.
\begin{definition}[Farey descendants]
 Let $(s,t)$ be a maximal extremal pair. Then we say the pairs $(s_{\rr}, t_{\rr})$ given by \begin{align*}
s_{\rr} &= \rho_{r_1} \rho_{r_2} \dots \rho_{r_n} (s), \\
 t_{\rr} &= \rho_{r_1} \rho_{r_2} \dots \rho_{r_n} (t).
\end{align*}
are the \emph{Farey descendants} of $(s,t)$.
\end{definition}
Because of the tree construction, the results stated above specifically for balanced descendants then hold for the Farey descendants of any maximal extremal pair, by simply repeating all proofs with $s$ in place of $0$ and $t$ in place of $1$. No further modification is needed. We will use this later to describe the boundary of $D_1(\beta)$.

\subsection{Bad $n$}

As in \cite{haresidorov14}, we say that a natural number $n$ is \emph{bad} for $(a,b)$ if every periodic point of period $n$ for $\Tb$ intersects the hole $(a,b)$.

In the $\beta=2$ case, it is natural to discard $n=2$ as there is only one periodic point of period $2$, making for an uninteresting definition. As $\beta$ decreases, gradually each $n$ will have fewer periodic points and thus once we have only one periodic point of period $n$ remaining we wish to discard this $n$.

To see the admissibility of a periodic orbit, we consider its largest point; that is to say the maximal cyclic permutation. For each $n$, the first (smallest) periodic orbit is $(10^{n-1})^\infty$. Then it is easy to see that the next periodic orbit is given by $(10^{k-2}10^k)^\infty$ for even $n=2k$ and by $(10^{k-1}10^k)^\infty$ for odd $n=2k+1$. Thus we should discard each $n$ at $\1 = (10^{k-2}10^k)^\infty$ for even $n=2k$ and at $\1 = (10^{k-1}10^k)^\infty$ for odd $n=2k+1$. For each $\beta \in (1,2)$ let $N_\beta$ denote the least $n$ such that there exists at least two $n$-cycles for $\beta$.

 Then let $B_\beta(a,b)$ denote the set of $n>N_\beta$ such that $n$ is bad for $\Tb$. Then we define
\begin{equation*}
D_2(\beta) = \{ (a,b) : B_\beta(a,b) \text{ is finite}\}.
\end{equation*}

\section{Transfer of results from the doubling map} \label{transfersection}
In this section we transfer what results we can from the case of the doubling map as studied by Glendinning and Sidorov in \cite{glendinningsidorov15} and Hare and Sidorov in \cite{haresidorov14}.

\begin{remark} \label{cornerremark}
 We remark that for each $i \in \{0,1,2\}$, if $(a,b) \in D_i(\beta)$ then $(a+\epsilon, b-\epsilon) \in D_i(\beta)$ also. This is simply because the latter hole is contained in the former, thus $\J(a,b) \subseteq \J(a+\epsilon, b-\epsilon)$.
\end{remark}

\subsection{Large and small $a$ and $b$}
For the doubling map one may restrict to $(a,b) \in (1/4,1/2) \times (1/2,3/4)$ without losing any interesting behaviour. This section covers the analogue of this restriction for $\beta \in (1,2)$.

\begin{lemma}[Large $a$]
 If $a>\frac{1}{\beta}$ then $(a,b) \in D_2(\beta) \cap D_1(\beta) \cap D_0(\beta)$.
\end{lemma}

\begin{proof}
 Just as in the $\beta=2$ case, write $a= 1 0^k 1 \dots$ for some $k \geq 0$. Then consider the following subshift:
\begin{equation*}
 A = \{ w \in \Xb : w_i = 1 \implies w_{i+j} = 0 \text{ for } j = 1, \dots, k+1 \}.
\end{equation*}
Then $A \subset \J(a,b)$. $A$ contains periodic orbits $(10^m)^\infty$ for any $m>k+1$, thus $(a,b)$ has only finitely many bad $n$. $A$ also has positive entropy, therefore positive Hausdorff dimension, so $(a,b) \in D_1(\beta)$ and $(a,b) \in D_0(\beta)$ also.
\end{proof}

The restriction for small $b$ is more different from the $\beta=2$ case as it involves $\gb$. 

\begin{remark} \label{smallbd0rmk}
 By definition of infimum, we have that $X_{\gb} \subseteq \J(0, \infXgb)$. Therefore $(a,b) \in D_0(\beta)$ whenever $b < \infXgb$.
\end{remark}

\begin{lemma}[Small $b$] \label{smallb}
Suppose $\gb \in \Q$ and $\gamma_1$ is the left Farey parent of $\gb$. Let $u_{\gb}$ and $u_{\gamma_1}$ denote the minimal shifts of the balanced words associated to $\gb$ and $\gamma_1$ respectively. Then $(a,b) \in D_2(\beta) \cap D_1(\beta) \cap D_0(\beta)$ whenever $b < u_{\gamma_1} u_{\gb}^\infty$.
\end{lemma}
\begin{proof}
Consider $b < u_{\gamma_1} u_{\gb}^\infty$. Then $b=u_{\gamma_1} u_{\gb}^K v$ for some $K$ and some $v\prec u_{\gb}$. Consider the following shift:
\begin{equation*}
B_K = \{ w \in \Xb : w \text{ is made of blocks } u_{\gamma_1} u_{\gb}^k  \text{ for any } k>K \}.
\end{equation*}
The least element of $B_K$ is $(u_{\gamma_1} u_{\gb}^{K+1})^\infty$. Therefore $B_K \subset \J(a,b)$. As $|u_{\gb}|$ and $|u_{\gamma_1}|$ are coprime and we are allowed any $k>K$, this shift $B_K$ will contain periodic orbits of any suitably long length. Thus $(a,b)$ has finitely many bad $n$. $B_K$ also has positive entropy, therefore positive Hausdorff dimension, so $(a,b) \in D_1(\beta)$ and $(a,b) \in D_0(\beta)$ also.
\end{proof}
\begin{corollary}
 Suppose $\gb \notin \Q$. Then $(a,b) \in D_2(\beta) \cap D_1(\beta) \cap D_0(\beta)$ whenever $b < \infXgb$.
\end{corollary}
\begin{proof}
 This follows by approximating $\gb$ from below by rationals. As $\gamma \nearrow \gb$, notice that the left Farey parent $\gamma_1$ of $\gamma$ will also tend to $\gb$. Therefore $u_{\gamma_1} u_\gamma^\infty \nearrow \infXgb$.
\end{proof}

We also note the following lemma:
\begin{lemma} \label{Xgbhole}
  $\J(0 \infXgb, \infXgb) \doteq X_{\gb}$.
\end{lemma}
\begin{proof}
We show this only for rational $\gb$ as the irrational case is largely similar. As per Remark \ref{smallbd0rmk}, we know $X_{\gb} \subseteq \J(0, \infXgb)$.
Consider $B_K$ as defined in Lemma \ref{smallb}. As $K \to \infty$, or equivalently as the right endpoint of the hole $(a,b)$ tends to $u_{\gamma_1} u_{\gb}^\infty$, $B_K \to { \sigma^n u_{\gb}^\infty} = X_{\gb}$. Therefore the hole $(0, u_{\gamma_1} u_{\gb}^\infty)$ has avoidance set essentially equal to $X_{\gb}$. We have $u_{\gamma_1} u_{\gb}^\infty < \infXgb$, so this implies that $\J(0, \infXgb) \doteq X_{\gb}$ also.

 This means we only need consider $x< 0 \infXgb$. In this case we have that either $x = 0^n \infXgb$ for some $n>1$, or $x \in (0^n \infXgb, 0^{n-1} \infXgb)$ for some $n>1$. Then $\sigma^{n-1}x \in (0 \infXgb, \infXgb)$ as required. Thus $\J(0 \infXgb, \infXgb) \doteq X_{\gb}$.
\end{proof}
\begin{corollary}
  $\J(1/\beta, 1 \infXgb) \doteq X_{\gb}$.
\end{corollary}
\begin{proof}
 This hole is a preimage of $(0, \infXgb)$ which by the previous lemma has avoidance set $X_{\gb}$.
\end{proof}

These results describe what are essentially the easy cases, where $a\geq 1/\beta$ or $b\leq \infXgb$. The interesting behaviour that is more difficult to describe thus occurs within this region for $(a,b)$:
\begin{equation*}
I_\beta = (0 \infXgb, 1/\beta) \times (\infXgb, 1 \infXgb).
\end{equation*}

Notice that as $\beta \to 2$, $I_\beta$ approaches as expected the $(1/4,1/2) \times (1/2,3/4)$ region seen for the doubling map. However as $\beta \to 1$, $I_\beta \to~(0,1)^2$.

\subsection{Extremal pairs}
We now commence to transfer results from the doubling map. Essentially, if an extremal pair is admissible for a given $\beta$, then all results involving that extremal pair will still hold for that $\beta$. We formalise this as follows.

\begin{theorem} \label{TransferLemma}
Let $\beta \in (1,2)$. Suppose $(s,t)$ is an extremal pair such that $\{s,t\}^\N \subset \Xb$ and $(s^\infty, t^\infty) \in I_\beta$. Let $u$ and $v$ be words such that $s=uv$ and $t=vu$. Then for any $\epsilon>0$, we have
\begin{enumerate}
\item $\J(s^\infty, t^\infty) \supseteq \{ \sigma^n s^\infty : n \geq 0 \}$,
\item $\J(s^\infty, ts^\infty - \epsilon)$ and $\J(st^\infty+\epsilon, t^\infty)$ are uncountable.
\end{enumerate}
If additionally $j=|u|$ and $q=|s|$ are coprime, then
\begin{enumerate}\addtocounter{enumi}{2}
 \item $(s^\infty, ts^\infty - \epsilon)$, $(st^\infty+\epsilon, t^\infty)$ and $(st^\infty+\epsilon, ts^\infty - \epsilon)$ have finitely many bad $n$.
\end{enumerate}
\noindent
If $(s,t)$ is a maximal extremal pair, then
\begin{enumerate}\addtocounter{enumi}{3}
\item $\J(s^\infty, t^\infty) \doteq \{ \sigma^n s^\infty : n \geq 0 \}$.
\end{enumerate}
\noindent
If $(s,t)$ is the Farey descendant of a maximal extremal pair, then
\begin{enumerate}\addtocounter{enumi}{4}
 \item $\J(s^\infty, t^\infty)$ is countable,
 \item $\J(sts^\infty, ts^\infty)$ and $\J(st^\infty, tst^\infty)$ are countable.
\end{enumerate}
\end{theorem}

\begin{proof}
These results are shown for the case of balanced pairs and their descendants in \cite{glendinningsidorov15} and \cite{haresidorov14}. We collect the results together here in a bid to make clearer precisely what combinatorial property of words each result is relying upon, and so for the sake of clarity we repeat the arguments here and alter them as necessary to encompass the general case.

Let $(s,t)$ be an extremal pair. Item (1) -- that $\{ \sigma^n s^\infty : n \geq 0 \} \subseteq \J(s^\infty, t^\infty)$ -- follows immediately from the definition.

For item (2), to show that $\J(s^\infty, ts^\infty - \epsilon)$ is uncountable, we follow \cite[Lemma 2.2]{glendinningsidorov15}. Let $N \in \N$ and define
\begin{equation*}
 W_N = \overline{\{ \sigma^i w : w \text{ is composed of blocks } u t^m \text{ with } m>N \}}.
\end{equation*}
Because $\{s,t\}^\N$ is admissible, we know that $W_N$ is admissible for all $N$. Furthermore $W_N$ is shift invariant and has positive entropy (and therefore positive Hausdorff dimension).
For any $\epsilon>0$, there exists an $N$ such that $ts^\infty - \epsilon <  ts^N$. Then we claim $W_N \subset \J(s^\infty, ts^\infty - \epsilon)$.
To see this, notice that $ut^m = s^mu$. By extremality, the only shifts we need be concerned about are those beginning with $s$ or $t$. Any shift beginning $s$ will be of the form $s^i uu \prec s^\infty$, so avoids the hole. Any shift beginning $t$ either has multiple $t$s and so avoid the hole, or begins $ts^muu$ for some $m>N$. This therefore also avoids the hole for large enough $N$.

The case $\J(st^\infty+\epsilon, t^\infty)$ is similar, using shifts with $t^mv = vs^m$. Either of these shifts will then avoid $(st^\infty+\epsilon, ts^\infty - \epsilon)$. This leads immediately to item (3), following \cite[Theorem 3.6]{haresidorov14}. Notice that the orbit $(ut^m)$ has period $m q + j$. If $j$ and $q$ are coprime, then for every $\ell$ there exists $k$ such that $\ell \equiv kj \mod q$. So by considering points of the form
\begin{equation*}
 w = (ut^{m_1}ut^{m_2} \dots ut^{m_k})^\infty,
\end{equation*}
for sufficiently large $m_i>N$, we can create orbits of any sufficiently large length which avoid the hole. Thus whenever $|u|$ and $|s|$ are coprime, we have that $(s^\infty, ts^\infty - \epsilon)$, $(st^\infty+\epsilon, t^\infty)$ and $(st^\infty+\epsilon, ts^\infty - \epsilon)$ have finitely many bad $n$.

Item (4) is a restatement of Lemma~\ref{maxisnice}.

For (5), we follow \cite[Lemma 2.12]{glendinningsidorov15} and use induction to show that $\J(s^\infty, t^\infty)$ is countable for Farey descendants of maximal extremal pairs. The result clearly holds for the maximal extremal pair itself. Assume the claim holds for all $k$th level descendants, meaning pairs $(s_k, t_k) = (s_{(r_1, \dots r_k)}, t_{(r_1, \dots r_k)})$. We show it must then hold for the $(k+1)$st level. Write $r_{k+1}=p_{k+1}/q_{k+1}$.
Note firstly that as $\J(s_k^\infty, t_k^\infty)$ is countable, we wish to show that all but countably many points of $(s_k^\infty, t_k^\infty)$ must fall into $(s_{k+1}^\infty, t_{k+1}^\infty)$.

Any word $(s_{k+1}, t_{k+1})$ is by definition a balanced word on the alphabet $\{s_k, t_k\}$, with length in this alphabet $q_{k+1}$. The only shifts of $s_{k+1}$ that fall into $(s_{k+1}^\infty, t_{k+1}^\infty)$ are those beginning with $s_k$ or $t_k$. Label these (in order) as $x_1, \dots, x_{q_{k+1}}$. Balanced words correspond to \emph{ordered} orbits as discussed in \cite{bullettsentenac94} and \cite{goldbergtresser96}. This means that any interval $[x_i, x_{i+1}]$ will be mapped by $\sigma^{q_{k+1}}$ to some other interval $[x_j, x_{j+1}]$, and by repeatedly applying $\sigma^{q_{k+1}}$ we will cycle through all possible $j\in \{1, \dots, q_{k+1}-1 \}$. One of these intervals is $[s_{k+1}^\infty, t_{k+1}^\infty]$. Therefore all but countably many points in $(x_1, x_{q_{k+1}})$ will fall into $(s_{k+1}^\infty, t_{k+1}^\infty)$.

The only remaining possibilities are points in $(s_k^\infty, x_1)$ and points in $(x_{q_{k+1}}, t_k^\infty)$. Applying $\sigma^{q_{k+1}}$ to these intervals maps them to $(s_k^\infty, x_i)$ and $(x_j, t_k^\infty)$ respectively for some $i>1$ and $j<q_{k+1}$. Thus again by applying $\sigma^{q_{k+1}}$ repeatedly we see that all but countably many points must fall into $(s_{k+1}^\infty, t_{k+1}^\infty)$.

Therefore item (5) holds and $\J(s^\infty, t^\infty)$ is countable for Farey descendants of maximal extremal pairs.

The final item (6) is more complex, with a degree of subtlety as to why the result holds only for Farey descendants, not for either all extremal pairs or only maximal extremal pairs.\footnote{Note we consider maximal extremal pairs to be Farey descendants of themselves, so this result \emph{does} hold for maximal extremal pairs.} We follow \cite[Theorem 2.13]{glendinningsidorov15}. Suppose $(s,t)$ is a Farey descendant of a maximal extremal pair $(u,v)$. Consider $\J(s^\infty, t^\infty)$. This is a countable subset of $\{u,v\}^\N$.

Then take any point in $[ts^\infty, t^\infty]$. By applying $\sigma^q$ repeatedly we can see that all but countably many points in this interval must fall into $(s^\infty, ts^\infty)$. Hence $\J(s^\infty, ts^\infty) \setminus \J(s^\infty, t^\infty)$ is countable. Therefore $\J(s^\infty, ts^\infty)$ is countable.

Similarly, consider any point in $[s^\infty, sts^\infty]$. Apply $\sigma^q$ repeatedly, and we see that all but countably many points must fall into $(sts^\infty, ts^\infty)$. Therefore we have that $\J(sts^\infty, ts^\infty) \setminus \J(s^\infty, ts^\infty)$ must be countable. Hence $\J(sts^\infty, ts^\infty)$ is countable.

The cases with $s$ and $t$ reversed are similar.
\end{proof}
\begin{remark}
It is key to note that item (6) relies very strongly on $\J(s^\infty, t^\infty)$ being countable, which holds only by item (5). This is where the proof fails if the pair $(s,t)$ is not a Farey descendant of a maximal extremal pair.
\end{remark}

The theorem gives the following corollary, as can be seen in Figure~\ref{maxextremalpic}.
\begin{corollary}
 $D_2(\beta) \subset D_1(\beta) \subset D_0(\beta)$.
\end{corollary}

\begin{figure}
\centering

\begin{tikzpicture}
\fill[color=gray] (0,0) -- (5,0) -- (5,5) -- (6,5) -- (6,-1) -- (0,-1) -- cycle;
\fill[color=lightgray] (0,0)--(0,5)--(5,5)--(5,1.5)--(3.5,1.5)--(3.5,0)--cycle;

\draw (-0.5,0) -- (5.5,0) node[right] {$ts^\infty$};
\draw (0,-0.5) node[below] {$s^\infty$} -- (0,5);
\draw (5,-0.5) node[below] {$st^\infty$} -- (5,5);
\draw (0,5) -- (5.5,5) node[right] {$t^\infty$};

\draw[dashed] (3.5,1.5) -- (3.5, -0.1) node[below] {$sts^\infty$};
\draw[dashed] (3.5,1.5) -- (5.1, 1.5) node[right] {$tst^\infty$};

\end{tikzpicture}

\caption{Lemma \ref{TransferLemma} for a maximal extremal pair $(s,t)$. The dark grey shows points in $D_2(\beta)$, the light grey shows points in $D_0(\beta)$, and the white region shows where the Farey descendants of $(s,t)$ will lie, so these points are in $D_0(\beta)$ and may or may not be in $D_1(\beta)$.}
\label{maxextremalpic}
\end{figure}
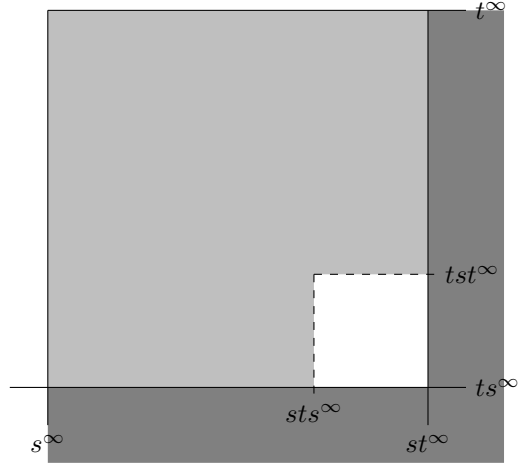

We now wish to establish which extremal pairs are maximal for a given $\beta$. The above results will then combine to delimit the boundaries of $D_0(\beta)$, $D_1(\beta)$, and $D_2(\beta)$, with minor modifications for cases where for example $s^\infty$ is admissible but $st^\infty$ is not. Each set has a continuous boundary consisting of a countable set of plateaus given by $[s^\infty, st^\infty]$ in the case of $D_0(\beta)$ and $D_2(\beta)$ and by $[s^\infty, sts^\infty]$ in the case of $D_1(\beta)$, as shown in Figure~\ref{maxextremalpic}. Recall Remark~\ref{fareydescs}: given a maximal pair $(s,t)$, we have -- up to a set of measure zero given by the limit points -- that
\begin{equation*}
 [s^\infty, st^\infty] = \bigcup_{(s_\rr, t_\rr)} [s_\rr^\infty, s_\rr t_\rr s_\rr^\infty],
\end{equation*}
where $(s_\rr, t_\rr)$ are the Farey descendants of $(s,t)$.

This ensures that once we have the correct maximal pairs, we can completely describe $D_1(\beta)$.

If it is the case that a pair $(s,t)$ is such that $s^\infty$ is admissible but $st^\infty$ is inadmissible, the above results are not significantly disrupted. Any inadmissible sequence must be replaced by the largest admissible sequence that is less than the intended inadmissible sequence. The results showing that a point is \emph{not} in $D_i(\beta)$ for some $i$ will clearly still apply as we have fewer admissible sequences meaning $\J$ will if anything be smaller than previously shown. The difficulty is when we want to show that $\J$ is large, as we must ensure the inadmissibility has not removed too much of $\J$. Fortunately it is easy to see that for the first parts of items (2) and (3) involving points $(s^\infty, ts^\infty-\epsilon)$, the relevant sequences will still be admissible. Then combining with Remark~\ref{cornerremark} we have the needed modification.
\begin{figure}

\centering

  \includegraphics[scale=0.58]{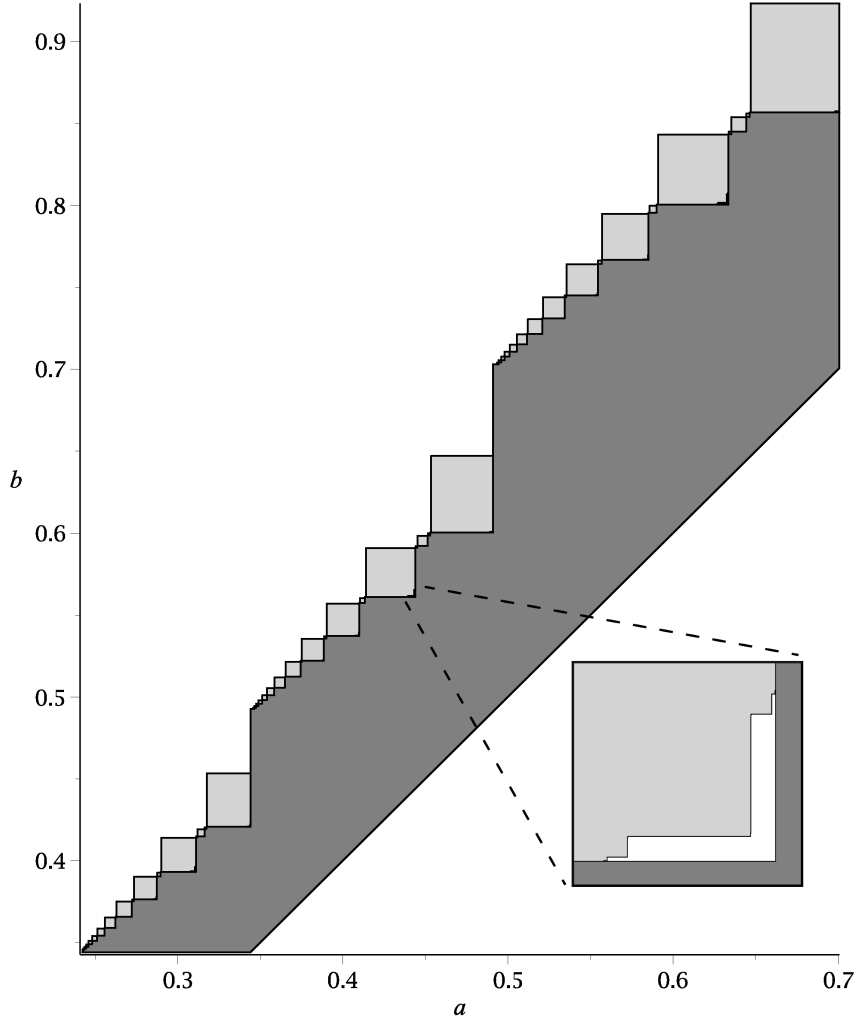}
\caption{$D_2(\beta)$ (dark grey), $D_1(\beta)$ (white + dark grey) and $D_0(\beta)$ (light grey + white + dark grey) for the region $I_\beta \setminus R_\beta$ with $\1 = (10010000)^\infty$, $\beta \approx 1.427$, $\bgb = (1/4, 1/2)$, with magnified area inset to clearly show $D_1(\beta)$.}
\label{Dibalanced}
\end{figure}
\begin{eg} \label{trunceg}
Consider $\1 = (10010000)^\infty$, $\beta \approx 1.427$, which has $\bgb = (1/4, 1/2)$. Then the balanced extremal pair given by $\gb = 1/4$ has $s_{1/4}^\infty = (0100)^\infty$ admissible but $s_{1/4}t_{1/4}^\infty=0100(1000)^\infty$ inadmissible. Then the greatest admissible sequence in $[s_{1/4}^\infty, s_{1/4}t_{1/4}^\infty]$ is $0(10010000)^\infty=10^\infty=1/\beta$.

Theorem~\ref{TransferLemma} tells us that usually the boundary of $D_0(\beta)$ should contain points $[s^\infty, st^\infty] \times t^\infty$, as seen in Figure~\ref{maxextremalpic}. Instead this interval will be truncated: the boundary of $D_0(\beta)$ will contain points $[s_{1/4}^\infty, 1/\beta] \times t_{1/4}^\infty$. Similarly the boundary of $D_2(\beta)$ will contain points $[s_{1/4}^\infty, 1/\beta] \times t_{1/4}s_{1/4}^\infty$.

For $D_1(\beta)$, the same effect will occur with the pair given by $\bgb=(1/4,1/2)$. Again the relevant interval will be truncated.

These effects can be seen in the top right of Figure~\ref{Dibalanced}: the point $(s_{1/4}^\infty, t_{1/4}^\infty)$ is approximately $(0.645, 0.92)$, and $1/\beta \approx 0.7$. As in Figure~\ref{maxextremalpic}, the other balanced pairs visibly give a light grey square showing $D_0(\beta)$, with $D_1(\beta)$ too small to see. The affect of the inadmissibility of $st^\infty$ is that instead there is a light grey rectangle.
\end{eg}

Considered the balanced pairs $(s,t) = (0$-$\max(w_\gamma), 1$-$\min(w_\gamma) )$ discussed in the previous section. As shown by Glendinning and Sidorov, these pairs (when admissible) are maximal extremal. Each pair is admissible when the associated $\gamma$ is less than $\gb$. Notice that in the context of this problem, we consider these particular pairs because they are maximal extremal: that they are balanced is a side effect, not the reason for interest.

For $\beta=2$, the holes formed from these pairs will customarily have two distinct preimages, formed by appending either a $0$ or a $1$ as a prefix to both endpoints of the hole. As $\beta$ decreases, the preimage formed by appending a $1$ becomes inadmissible and so a particular hole may have a unique preimage. Because $\J(a,b)$ is invariant under $\Tb$, this means that all results pertaining to the original hole will also apply to its unique preimage. This leads us to the following conclusion:

\begin{lemma} \label{shiftbalanced}
 Suppose $\gb \in [\frac{1}{n+1}, \frac{1}{n})$ and $0 < k \leq n$. Let $(s,t)=(0$-$\max(w_\gamma), 1$-$\min(w_\gamma)) $ be the maximal extremal balanced pair corresponding to $\gamma < \gb$. Then the pairs $(0^k$-$\max(w_\gamma), 0^{k-1}1$-$\min(w_\gamma))$ are also maximal extremal.
\end{lemma}
Notice that when $\gb \in [\frac{1}{n+1}, \frac{1}{n})$, we have that $\1$ begins with $10^{n-1}$. This gives the correct range of $0 < k \leq n$ to ensure a unique preimage. Also note that $\infXgb$ begins with $0^n$, so as one would expect these pairs will fall into the region $I_\beta=(0 \infXgb, 1/\beta) \times (\infXgb, 1 \infXgb)$.

This means that by considering all suitable $\gamma$ and $k$, balanced pairs will cover the range $(0^n1, 0\1) \times (0^{n-1}1, 1 \infXgb)$. However these are all available balanced pairs, and so we cannot expect the remaining region $R_\beta =(0 \infXgb, 0^n1) \times (\infXgb, 0^{n-1}1)$ to involve balanced pairs.

Figure \ref{Dibalanced} shows the balanced pairs giving $D_i(\beta)$ for $I_\beta \setminus R_\beta$, with $\beta \approx 1.427$. Note that $D_1(\beta)$ is shown by the dark grey and the white areas ``between'' the light and dark grey. These white areas do exist but are so small as to be barely visible, therefore the inset image shows a magnification as indicated. Notice how the overall image has the same section repeated three times at different scales. This corresponds to the shifting of the balanced words as in Lemma~\ref{shiftbalanced} above. Furthermore there are vertical intervals that appear to be jumps, at $a = 1/\beta^k$, such that $\partial D_2(\beta) = \partial D_1(\beta) = \partial D_0(\beta)$. This corresponds to where $s^\infty$ is admissible but $st^\infty$ is inadmissible, as discussed in Example~\ref{trunceg}.

\section{The region $R_\beta=(0 \infXgb, 0^n1) \times (\infXgb, 0^{n-1}1)$} \label{rbeta}
The previous sections have described $D_0(\beta)$, $D_1(\beta)$ and $D_2(\beta)$ for $a \leq 0 \infXgb$ and for $a \geq 0^n 1$. In countably many cases the remaining region $R_\beta$ is empty. This occurs precisely when $\1 = w_\gamma^\infty$ for
\begin{equation*}
 \gamma = \frac{k}{(n+1)k-1} \in (1/(n+1), 1/n],
\end{equation*}
with $k$, $n \in \N$. For these values of $\beta$, the description of the $D_i(\beta)$ is already complete and needs only balanced pairs. The doubling map is one of these exceptional cases ($n=k=1$), as is the golden ratio $\beta = (1+\sqrt{5})/2$ ($n=2, k=1$).

For the remaining $\beta$ we hence need to find the maximal extremal pairs that fall into the region $R_\beta$.

As it happens, the required extremal pairs for $R_\beta$ are defined using balanced words as per the following algorithm. For some rational $\gamma \in (1/(n+1), 1/n)$, denote by $u_\gamma$ the \emph{minimal} shift of the balanced word formed from $\gamma$. Then create a Farey-like tree of words beginning with $0$ and $u_\gamma$ as the roots of this tree. Just as in a standard Farey tree, neighbours $u_1$ and $u_2$ such that $u_1^\infty < u_2^\infty$ are combined to form $u_2u_1$.

As $\gamma \in (1/(n+1), 1/n)$, the resultant words $w$ (excepting roots $0$ and $u_\gamma$) must contain $0^{n+1}$ as a factor. Therefore, define $s=0^{n+1}$-$\max (w)$ and $t=0^n1$-$\min (w)$. These by definition form an extremal pair. We claim that for the right combinations of $\gamma$ and $\beta$, these are the maximal pairs. This will be illustrated by Example \ref{rexample} at the end of the section and is shown in Figure~\ref{rpic}.
\begin{figure}[t]

\centering

  \includegraphics[scale=0.58]{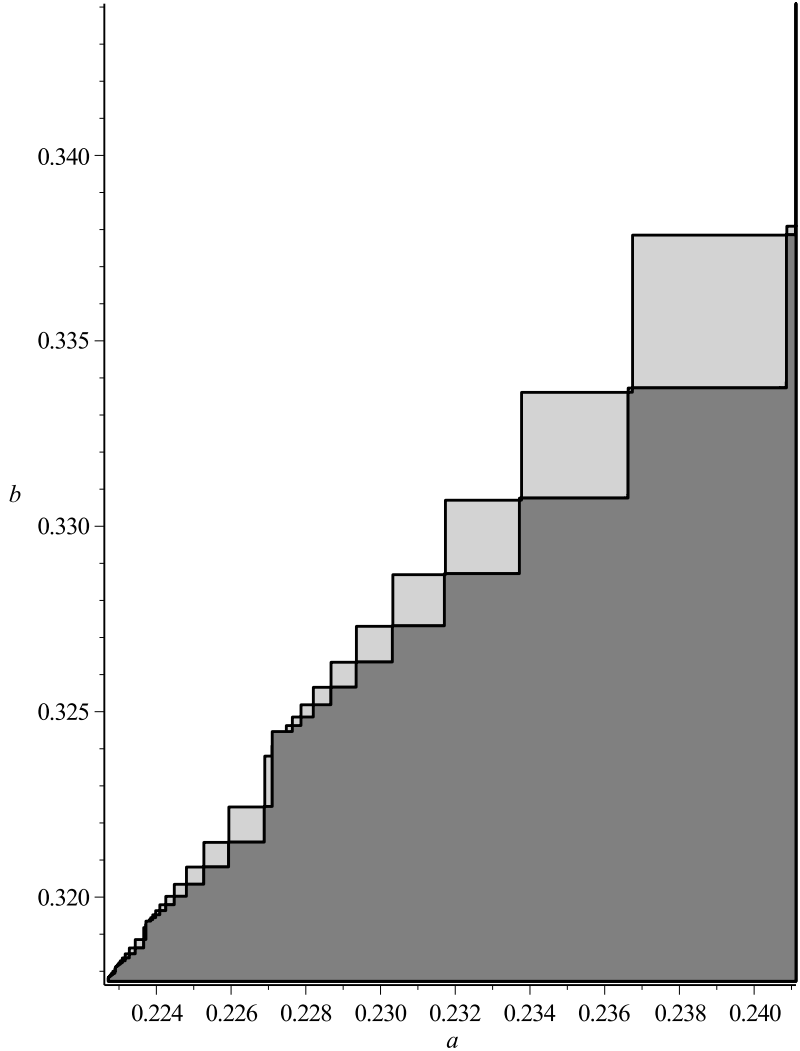}
\caption{$D_2(\beta)$ (dark grey), $D_1(\beta)$ (white + dark grey) and $D_0(\beta)$ (light grey + white + dark grey) for the region $R_\beta$ with $\1 = (10010000)^\infty$, $\beta \approx 1.427$, $\bgb = (1/4, 1/2)$.}
\label{rpic}
\end{figure}

We begin by showing that if the pairs defined above fall into $R_\beta$, then they must be maximal. We do this by induction, exploiting the tree structure of the definition.
\begin{lemma} \label{treeworks}
 For any minimal cyclically balanced $u_\gamma$ associated to $\gamma=p/q$, consider the pair $(s,t) = (0u_\gamma^k, u_\gamma 0u_\gamma^{k-1})$, $k \geq 1$. If $\beta$ satisfies $(s^\infty, t^\infty) \in R_\beta$, then $\J[s^\infty, t^\infty]=\emptyset$: that is to say, $(s,t)$ is a maximal extremal pair.
\end{lemma}
\begin{proof}
 Consider a point $x \in (0,1)$. We know by Lemma~\ref{Xgbhole} that
\begin{equation*}
\J[0 \infXgb, \infXgb] = \emptyset,
 \end{equation*}
 so we may restrict to $x \in [0 \infXgb, \infXgb]$. This region overlaps the hole under consideration, so restrict $x$ to $[0 \infXgb, s^\infty] = [0 \infXgb, (0u_\gamma^k)^\infty]$. By applying the shift map we may restrict to $[\infXgb, (u_\gamma^k0)^\infty]$. Again this overlaps the hole, so restrict to $[t^\infty, \sigma s^\infty] = [(u_\gamma0u_\gamma^{k-1})^\infty, (u_\gamma^k0)^\infty]$. Then apply $\sigma^q$ repeatedly and restricting each time, we may conclude that any point still remaining must itself be equal to $t^\infty$.
\end{proof}

\begin{remark}
 Following the above proof it is also easy to see that pairs $(s,t) = (0u_\gamma0^{k-1}, u_\gamma0^k)$ must also be maximal extremal: notice that in this case it is simpler as $t^\infty = \sigma s^\infty$.
\end{remark}

\begin{lemma} \label{treekidswork}
 Suppose two maximal extremal pairs $(s_1, t_1)$ and $(s_2, t_2)$ are Farey neighbours in a tree generated by $0$ and $u_\gamma$ for a minimal cyclically balanced word $u_\gamma$. Then the pair $(s,t) = (s_2s_1, t_1t_2)$ is also maximal.
\end{lemma}
\begin{proof}
This has been shown for the tree of balanced words in \cite[Lemma 2.5]{sidorov14} and again is more a property of the tree construction than of the specific words. We repeat the proof here for completeness' sake. To show maximality of $(s,t)$, we aim to show that $\J[s^\infty, t^\infty] = \emptyset$. Consider $x \in (0,1)$. We know by maximal extremality that $\J[s_1^\infty, t_1^\infty] = \emptyset$, so we may restrict to $x \in [s_1^\infty, t_1^\infty]$. We know by the tree construction that $s_1^\infty < s^\infty$ and $t_1^\infty < t^\infty$, so we may restrict to $x \in [s_1^\infty, s^\infty)$. Then $\sigma^{|s_1|}(x) \in [s_1^\infty, t_1^\infty]$, so restrict again. Continuing this process, we see that the only possible point avoiding $[s^\infty, t^\infty]$ must be $s_1^\infty$. But then this shifts to $t_1^\infty \in [s^\infty, t^\infty]$.
\end{proof}

The above two lemmas combine to imply that if $(s,t)$ is a suitable admissible extremal pair as described, then $(s,t)$ is maximal extremal.

We now discuss which $\gamma$ are associated with which $\beta$. We describe this in two ways: firstly, by giving the set of correct $\beta$ for a particular $\gamma$ and secondly by giving the correct $\gamma$ in terms of a particular $\beta$.

\begin{lemma}
 Let $u_\gamma$ denote the minimal cyclic shift of the balanced word associated to $\gamma \in \Q$ with $\gamma \neq 1/n$ and left Farey parent $\gamma_1$. Then we have that the admissible pairs $(s,t)$ from the Farey tree formed by $0$ and $u_\gamma$ are maximal extremal pairs for $\beta$ if and only if $\gb \in [\gamma_1, \gamma)$.
\end{lemma}
\begin{proof}

  Firstly, notice that whenever $\gb \in [\gamma_1, \gamma)$, at least part of the Farey tree generated by $0$ and $u_\gamma$ will be admissible and give pairs $(s,t)$ satisfying $(s^\infty, t^\infty) \in R_\beta$. Outside of these values of $\gb$ we have that either the entirety of the tree will be inadmissible or the sequences will fall below $R_\beta$.

As $\1$ increases towards $w_\gamma^\infty$, more and more of the tree from $0$ and $u_\gamma$ becomes admissible. Therefore for every $k$, there exists $\beta$ such that the pair $[0 u_\gamma^k, u_\gamma0u_\gamma^{k-1})$ is an admissible extremal pair with $(s^\infty,t^\infty) \in R_\beta$, and so in the limit the entire tree gives maximal extremal pairs.
\end{proof}
\begin{figure}[t]
 
\begin{tikzpicture}[->]
\node (0) at (0,5) {$0$};
\node (1) at (10,5) {$0001001$};

\node (12) at (5,4) {$00010010$};

\node (13) at (3.3,3) {$000100100$};
\node (23) at (6.7,3) {$000100100010010$};

\path (0) edge (12);
\path (0) edge (13);

\path(1) edge (12);
\path(1) edge (23);

\path (12) edge (13);
\path (12) edge (23);
\end{tikzpicture}

\caption{Beginnings of a Farey tree with roots $0$ and $u_{2/7}$.}
\label{nastyegpic}
\end{figure}
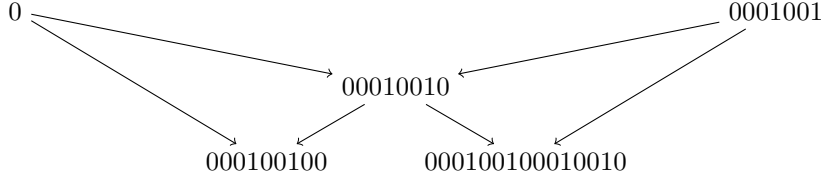
\begin{lemma} \label{whichgamma}
 Let $\gb \in \Q$ with continued fraction expansion $[0; a_1, a_2, \dots, a_n]$, with $a_n > 1$. Then $\gb \in [\gamma_1, \gamma)$ if and only if $\gamma$ has continued fraction expansion given by $[0; a_1, \dots, a_{2k+1}]$ with $2k+1 < n$ or by $[0; a_1, \dots, a_n-1, 1, k]$ for $k \geq 1$ if $n$ is odd and $[0; a_1, \dots, a_n, k]$ for $k \geq 1$ if $n$ is even.
\end{lemma}
\begin{proof}
 As is well known, the odd convergents in the continued fraction expansion of a number give a decreasing sequence of overapproximations of that number. It is also well known that given a rational $\gamma$ with continued fraction expansion $[0; b_1, \dots, b_n]$ with $b_n>1$, its Farey parents are given by $[0; b_1, \dots, b_{n-1}]$ and $[0; b_1, \dots, b_n-1]$. If $n$ is odd, then $[0; b_1, \dots, b_{n-1}]$ must therefore be the left Farey parent. Therefore it follows immediately that the odd convergents of $\gb$ will satisfy $\gb \in [\gamma_1, \gamma)$. Those given by $[0; a_1, \dots, a_n-1, 1, k]$ for $n$ odd or $[0; a_1, \dots, a_n, k]$ for $n$ even are then further overapproximations.
\end{proof}

The case $\gamma=1/n$ is excluded simply because it becomes subsumed in other cases: for example, the word $000101$ may be consider as $00(01)^2$ with $\gamma = 1/2$ or as $0(00101)$ with $\gamma =2/5$.

It remains to explain why these intervals cover almost all of $R_\beta$. To see this, suppose that $\gb \in [\gamma_1, \gamma)$, or equivalently
\begin{equation} \label{condforrbeta}
  \1 \in [(w_{\gamma_1})^\infty, w_\gamma^\infty).
\end{equation}

Then for almost every $\beta$ in this range there exists a maximal pair $(s,t)$ in the tree from $0$ and $u_\gamma$ such that $s^\infty$ is admissible and $st^\infty$ is inadmissible. Then consider the greatest admissible sequence in $[s^\infty, st^\infty]$. This will end in $\1$ so may be rewritten as a finite sequence.

\begin{lemma} \label{pointsfit}
 The greatest admissible finite sequence for $\gamma$ described above is equal to $0u_{\gamma_2}$.
\end{lemma}
\begin{proof}
 $s^\infty$ is admissible so clearly the sequence $st^\infty$ becomes inadmissible with the very first $t$. Therefore, to be admissible we should truncate from the maximal shift of $s$ and replace the preceding $0$ with a $1$. We know that $s$ begins $0u_\gamma$ and must be the maximal shift beginning this way, so consider $u_\gamma = u_1 \dots u_q = u_{\gamma_1} u_{\gamma_2}$. There exists $k$ such that $\sigma^k u_\gamma$ begins $w_{\gamma_2}$, which will be the point at which to truncate the sequence. Then because $u_\gamma$ is balanced, we have that $u_1 \dots u_{k-2}1 = u_{\gamma_2}$.\footnote{This follows from Lothaire \cite[Proposition 2.2.2]{lothaire02}.} This proves the lemma.
\end{proof}

The cases where $s^\infty$, $st^\infty$ or a limit point of the tree end in $\1$ and so may themselves be rewritten as finite sequences are similar.

The descendants of the above maximal pairs will be given by taking a balanced descendant pair $(s_{\rr}, t_{\rr})$ and applying the map $m: 0 \to 0$, $1 \to u_\gamma$. This completes the description of $D_1(\beta)$.

\begin{eg} \label{rexample}
 Consider $\1 = (10010000)^\infty$, or equivalently $\beta \approx 1.427$, which has $\bgb = (1/4, 1/2)$. In this case we have
\begin{equation*}
 R_\beta = (0(0001)^\infty, 0001) \times ((0001)^\infty, 001).
\end{equation*}

From Lemma \ref{whichgamma}, the correct $\gamma$ to consider are $\gamma_k=(k+1)/(4k+3)$ for $k \geq 1$.

Taking $k=1$, $\gamma=2/7$ gives $u_{2/7} = 0001001$. Create a tree with left root $0$ and right root $u_\gamma$, as shown in Figure~\ref{nastyegpic}.

The right half of this tree is inadmissible, but the left half gives admissible words. The largest admissible pair from this tree is the middle pair $(s,t) = (00001001, 00010010)$. Here we can see from $\1$ that $(s^\infty, t^\infty)$ is only just admissible and equals $(0001, 001)$, which is precisely the right endpoints of $R_\beta$. The smallest pairs given from $\gamma=2/7$ will be of the form $(0u_{2/7}0^{k-1}, u_{2/7}0^k)$ with $k$ large. Thus $\gamma=2/7$ enables us to cover the region $(00001001, 0001) \times (0001001, 001)$ with maximal extremal pairs. This region corresponds to approximately $(0.227,0.245) \times (0.324, 0.344)$ and can be clearly distinguished in Figure \ref{rpic}.

The next value of $\gamma$ is $3/11$, with $u_{3/11} = 00010001001$. Following the same process of creating a tree, the largest admissible pair from this tree is again be the middle pair, now given by $(s,t) = (000010001001, 000100010010)$.

Consider $st^\infty = 000010001001(000100010010)^\infty$. This is inadmissible. Truncating this sequence by choosing the largest admissible sequence in $[s^\infty, st^\infty]$ gives $00001001$. This illustrates Lemma \ref{pointsfit}: this largest possible sequence from $\gamma=3/11$ is precisely the limit of the smallest sequences from $\gamma=2/7$. If we were to continue to examine further values of $\gamma$, this pattern would continue, in that the maximal extremal pairs given from different $\gamma$ do not overlap one another and fit together perfectly leaving no gaps. The pairs from $\gamma = 3/11$ can also be seen in Figure \ref{rpic} and correspond to the region given by approximately $(0.2238, 0.227) \times (0.319, 0.324)$.

\end{eg}

Figure~\ref{rpic} shows an approximation to $D_i(\beta)$ in the region $R_\beta$ for $\beta \approx 1.427$. For this value of $\beta$, the region $R_\beta$ is very small, and $D_1(\beta)$ is once again too small to see distinctly.

\section{Summary}

We summarise the results in the following theorem.

\begin{theorem}
 Let $\beta \in (1,2)$ satisfy $\1 \in [1/(n+1), 1/n)$. Then as depicted in Figure~\ref{maxextremalpic}, we have that
\begin{itemize}
\item For any $a > 1/\beta$, $(a,b) \in D_2(\beta)$;
\item For any $b< \infXgb$, $(a,b) \in D_2(\beta)$;
 \item The boundary of $D_0(\beta)$ is given by joining points $(s^\infty, ts^\infty)$, $(s^\infty, t^\infty)$ and $(st^\infty, t^\infty)$ where $(s,t)$ are given by maximal extremal pairs for $\beta$;
 \item The boundary of $D_1(\beta)$ is given by joining the points $(s^\infty, ts^\infty)$ to $(sts^\infty, ts^\infty)$ and the points $(st^\infty, tst^\infty)$ to $(st^\infty, t^\infty)$, where $(s,t)$ are given by Farey descendants of maximal extremal pairs for $\beta$;
 \item The boundary of $D_2(\beta)$ is given by joining points $(s^\infty, ts^\infty)$, $(st^\infty, ts^\infty)$, and $(st^\infty, t^\infty)$ where $(s,t)$ are given by maximal extremal pairs for $\beta$.
\end{itemize}

The maximal extremal pairs for $\beta$ are given by
\begin{itemize}
 \item Shifts of the balanced word $w_\gamma$ with $\gamma \leq \gb$, given by $(0^k$-$\max(w_\gamma), 0^{k-1}1$-$\min(w_\gamma))$ for $0 < k \leq n$;
 \item Pairs $(0^n$-$\max(w), 0^{n-1}1$-$\min(w))$ formed from admissible words $w$ taken from a Farey tree with roots $0$ and $u_\gamma$, where $u_\gamma$ is the minimal shift of the balanced word corresponding to $\gamma \neq 1/k$ and $\gamma$ satisfies $\gb \in [\gamma_1, \gamma)$.
\end{itemize}

If a maximal extremal pair $(s,t)$ has $s^\infty$ admissible and $st^\infty$ inadmissible, then the above results hold with any inadmissible sequences replaced by the greatest admissible sequence in $[s^\infty, st^\infty]$.
\end{theorem}

In summary, the boundaries of $D_i(\beta)$ consist of a countable set of plateaus which are closely linked to the set of maximal extremal pairs for $\beta$, as explained in Section~\ref{transfersection}. The maximal extremal pairs are mainly balanced words, but for almost every $\beta \in (1,2)$ there is a small region where there are no admissible balanced words. In this region the maximal extremal pairs are formed by taking certain inadmissible balanced words and adding a $0$ to make them admissible, as described in Section~\ref{rbeta}. We include some pictures of $D_i(\beta)$ for different values of $\beta$; see Figures~\ref{wholeone} and \ref{wholetwo}.

\begin{figure}
\centering
  \includegraphics[scale=0.6]{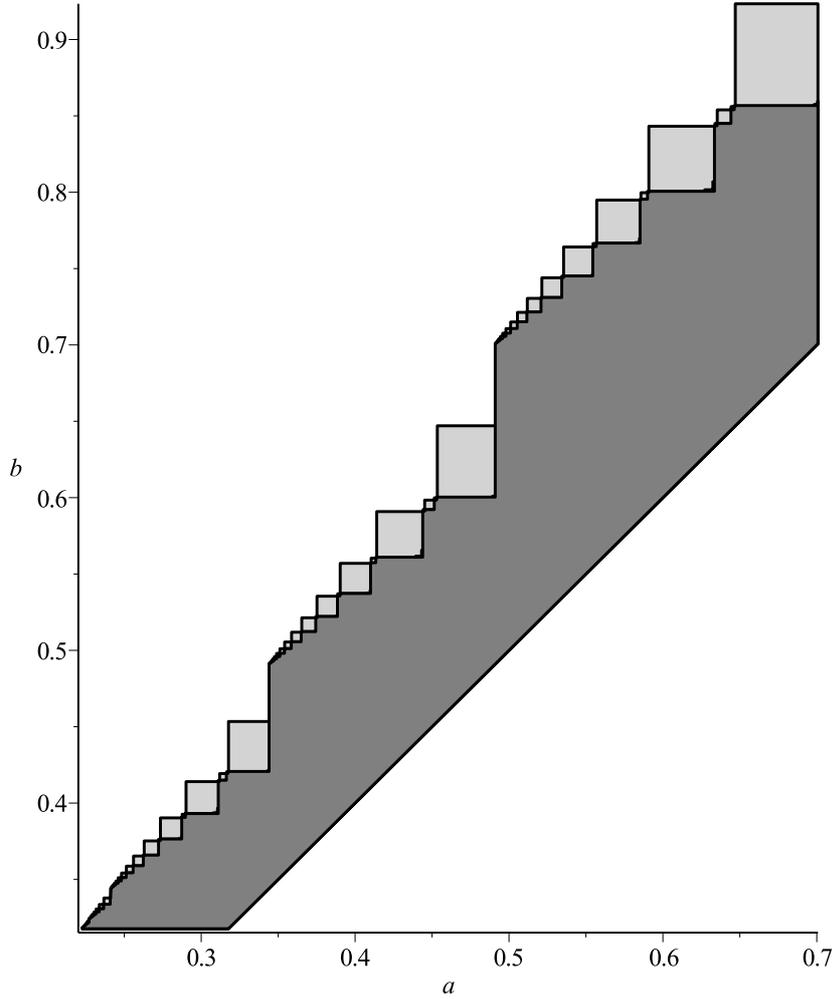}
\caption{$D_2(\beta)$ (dark grey), $D_1(\beta)$ (white + dark grey) and $D_0(\beta)$ (light grey + white + dark grey) for the region $I_\beta$ with $\1 = (10010000)^\infty$, $\beta \approx 1.427$, $\bgb = (1/4, 1/2)$.}
\label{wholeone}
\end{figure}
\begin{figure}
\centering
  \includegraphics[scale=0.58]{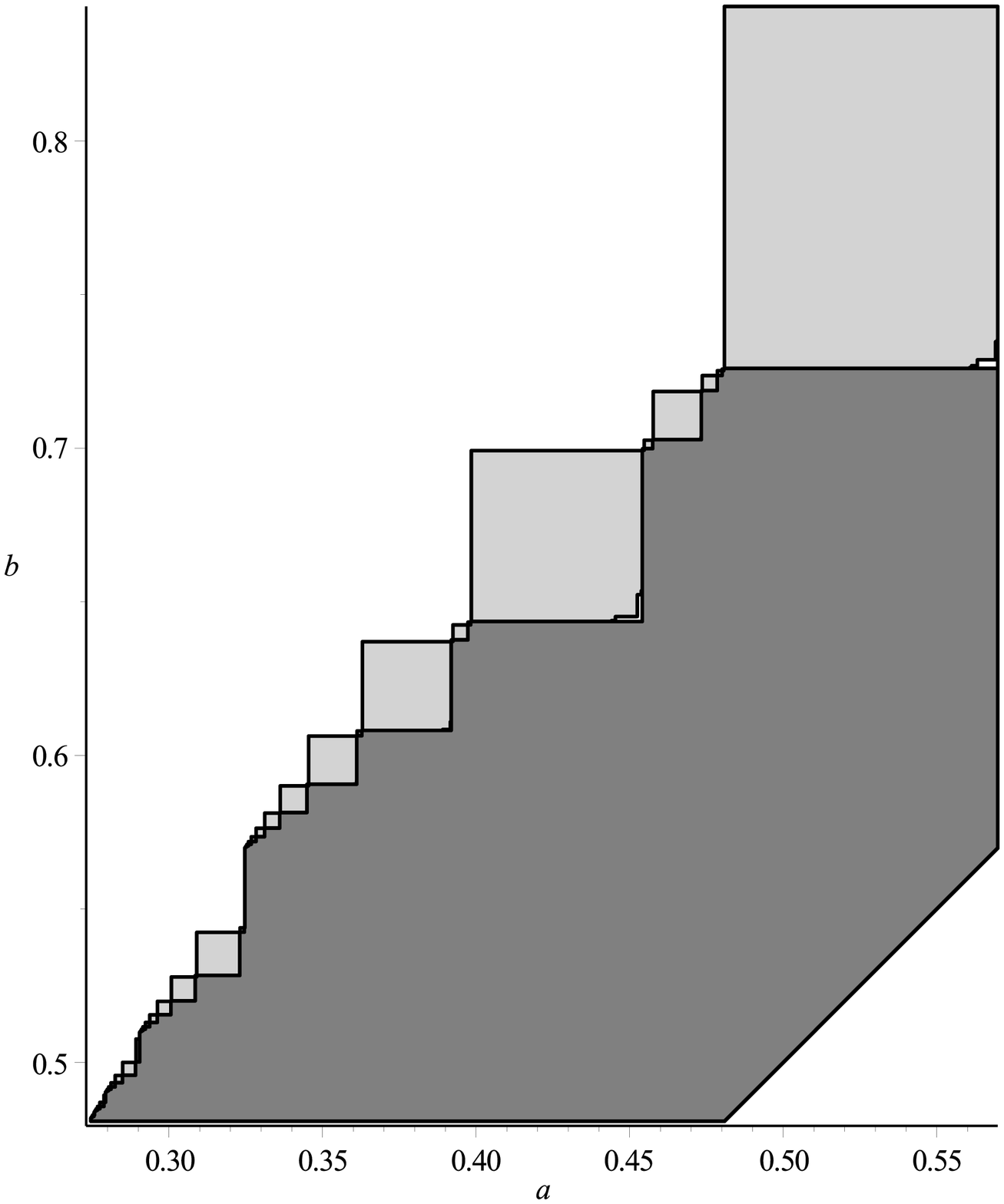}
\caption{$D_2(\beta)$ (dark grey), $D_1(\beta)$ (white + dark grey) and $D_0(\beta)$ (light grey + white + dark grey) for the region $I_\beta$ with $\1 = (1100)^\infty$, $\beta \approx 1.755$, $\bgb = (1/2, 1/2)$.}
\label{wholetwo}
\end{figure}

We note the following result.
\begin{lemma} Let $\gb \in [1/(n+1), 1/n)$. If $(a,b) \in D_i(\beta)$ then $b-a < C_i(\beta)$ where
\begin{align*}
C_0(\beta) &= \frac{\beta^{n-1}(\beta-1)}{\beta^{n+1}-1}, \\
C_1(\beta) = C_2(\beta) &= \frac{\beta-1}{\beta^2}. \\
\end{align*}
\end{lemma}
\begin{proof}
Each $C_i(\beta)$ is given by
\begin{equation*}
\sup \{ b-a : (a,b)\in D_i(\beta) \} = \max \{ b-a : (a,b) \text{ is a corner of }D_i(\beta) \}.
\end{equation*}
For $D_1(\beta)$ and $D_2(\beta)$ this is equal to
\begin{equation*}
\max \{ ts^\infty - s^\infty : (s,t) \text{ is maximal extremal for }\beta \}.
\end{equation*}
Then $ts^\infty - s^\infty = t-s$. Recall that $t = u 1$-$\min(w)$ and $s = u 0$-$\max(w)$ for some word $w$ with $u1$ and $u0$ factors of $w$. Therefore $t-s$ will be maximised when $u$ is the empty word. Then the only maximal extremal pairs with $s$ beginning $0$ and $t$ beginning $1$ are balanced pairs. For any balanced pair, we have $t = 10w$ and $s= 01w$ for some word $w$. Thus $t-s = 10-01 = 1/\beta - 1/\beta^2$, giving the required result for $C_1(\beta)$ and $C_2(\beta)$.

For $C_0(\beta)$, we need to maximise $t^\infty - s^\infty$ and so similarly may conclude that $t=10w$ and $s=01w$. This quantity will clearly be maximised when $w$ is as short as possible. Therefore given $\gb \in [1/(n+1), 1/n)$, we see that this will be maximised when $(s,t) = (10^n, 010^{n-1})$ as this is the shortest possible admissible $w$. Then $t^\infty - s^\infty$ gives $C_0(\beta)$ as stated.
\end{proof}
\section*{Acknowledgments}
The author would like to thank their supervisor Nikita Sido\-rov for his initial suggestion of the problem and his ongoing support, and Kevin Hare for kindly providing his code as a basis to help produce the figures.

\bibliographystyle{abbrv}
\bibliography{Bibliography}
\end{document}